\documentclass[12pt]{article}

\usepackage{amsmath,amsthm,amssymb,graphicx}

\def\real{\mathbb R}
\def\complex{\mathbb C}

\def\graph{{\mathcal G}}
\def\vset{\mathcal V}
\def\eset{\mathcal E}
\def\vertspace{{\mathbb V}}
\def\edgespace{{\mathbb E}}
\def\tree{\mathcal T}

\def\dop{{\mathcal L}}
\def\domain{{\mathcal D}}

\newtheorem{thm}{Theorem}[section]

\newtheorem{lem}[thm]{Lemma}
\newtheorem{prop}[thm]{Proposition}

\theoremstyle{definition}

\theoremstyle{remark}

\newcommand{\thmref}[1]{Theorem~\ref{#1}}

\newcommand{\lemref}[1]{Lemma~\ref{#1}}
\newcommand{\propref}[1]{Proposition~\ref{#1}}

\numberwithin{equation}{section}

\title{The Spectral Geometry of Biregular Graphs: \\ a Quantum Graph Approach}

\author{Robert Carlson
\\Department of Mathematics
\\University of Colorado at Colorado Springs
\\Colorado Springs, Colorado USA
\\rcarlson@uccs.edu}

\date{2023}

\begin{document}

\maketitle

\begin{abstract}
The spectral theory of the Laplace differential operator for biregular quantum graphs is developed.
Trees are studied in detail.  Generating functions for closed nonbacktracking walks
appear when resolvents for trees are related to resolvents for biregular graphs they cover.
The relationship between resolvent traces for finite graphs and walk generating functions is
especially productive.  A detailed description of the rational extension of nonbacktracking walk generating functions
is presented.  

\end{abstract}



\newpage
\section{Introduction}

Let $\graph $ denote a locally finite connected simple graph with a countable vertex set $\vset$
and edge set $\eset$.    
Edges are treated as intervals of length $1$.  As a consequence of this identification 
one may use the Hilbert space $L^2(\graph ) = \oplus _{e \in \eset} L^2(e)$ and 
a standard construction to equip $\graph $ with a self-adjoint differential operator
$\dop $ which acts by $-D^2$ on its domain.
 
This work treats the spectral theory of $\dop $ when $\graph $ is a biregular graph, that is a bipartite
graph whose vertex degrees are the same for all vertices in one of the two vertex classes.
Following the plan of \cite{Carlson97}, where regular graphs were studied, 
an essential part of the analysis is the spectral theory of $\dop $ for the biregular tree $\tree$,
including a detailed description of the spectrum and resolvent $R_{\tree}(\lambda ) = (\dop - \lambda I)^{-1}$. 
The biregular trees are universal covering spaces for biregular graphs.  This allows 
resolvents $R_{\graph }(\lambda ) = (\dop - \lambda I)^{-1}$ for graphs $\graph$ to be constructed using 
the tree resolvents.

When $\graph $ is finite, the resolvent  $R_{\graph }(\lambda )$ is a trace class operator.  Two expressions for 
the trace of the resolvent are developed.  The first is a conventional series based on eigenvalues.  The second expression
describes the trace as a product of a term coming from $\tree $ and a term involving
a generating function for closed nonbacktracking walks in $\graph $.
Given two finite biregular graphs with the same covering tree, the quotient of their resolvent traces is equal to 
a quotient involving their walk generating functions.  This formula is combined with complex variable methods to show that the  
walk generating functions extend as rational functions, with considerable information about pole numbers and locations.
 
The work starts with a discussion of quantum graphs and the differential operator $\dop $.   
For the finite graphs studied here, the eigenvalues of $\dop $ can be computed from the eigenvalues of a 
conventional discrete graph Laplacian $\Delta $ together with some additional combinatorial data.
The regular behavior of the eigenvalue sequence of $\dop $ allows for a compact description
of the trace of $R_{\graph}(\lambda )$ with explicit dependence on the eigenvalues of $\Delta $.   
 
The next section introduces biregular graphs.  Important examples for this work are the complete bipartite graphs
whose resolvent trace is computed.  Conveniently, from the quantum graph perspective regular graphs can be 
converted to biregular graphs with minimal effect on the spectrum, allowing this work to subsume some of the previous
analysis of regular graphs.  
 
The spectral theory for a biregular tree $\tree$ is then developed.
Given an edge $e$ in $\tree$, there are special solutions of the eigenvalue equation $-D^2y = \lambda y$
which are functions of the distance from $e$.  The propagation of these solutions depends on multipliers $\mu ^{\pm}(\lambda )$     
whose behavior largely determines the spectrum of $\dop $ for the tree.  The functions $\mu ^{\pm}(\lambda )$ have a central role 
in the construction of the resolvents $R_{\tree}(\lambda )$; 
their properties as analytic functions are discussed.  

The final section opens with a discussion of closed nonbacktracking walks in finite graphs and their generating functions.
The generating functions for complete bipartite trees are computed.   This is followed by material on covering spaces and the
construction of resolvents $R_{\graph}(\lambda )$ from $R_{\tree}(\lambda )$.  When $\graph $ is a finite biregular graph
the resolvent traces are closely related to walk generating functions, which are described in detail.  

In the case of regular graphs, where each vertex has the same degree, connections between the spectral theory of the discrete Laplacian and
such generating functions were previously explored in \cite{Brooks91}, and later from a quantum graph viewpoint in \cite{Carlson97}.  
A more recent and extensive treatment of related material is \cite{Terras}.  Spectral gaps for biregular graphs have also been treated;
see \cite{BDH} and the references given there. 
Many years ago Bob Brooks suggested that the techniques of \cite{Carlson97} could be used in the biregular case.  
His suggestion is gratefully acknowledged. 

\section{Differential operators on graphs}

Background material on quantum graphs \cite{BK}, tailored to our purposes, will be used.
In this work edges have length one, and are usually identified with $[0,1]$.
By using the identification of edges with intervals, function spaces and differential operators may be defined.  
Let $L^2(e) = L^2[0,1]$ be the usual Lebesgue space. $L^2( \graph )$ denotes the Hilbert space $\oplus _{e \in \eset} L^2(e)$.  
For $f \in L^2({\graph })$, the function $f_e:[0,1] \to \complex$ is the restriction of $f$ to the edge $e$.
The inner product is
\[\langle f, g \rangle = \int_{\graph } f\overline g
= \sum_{e \in \eset} \int_0^1 f_e(x)\overline{g_e(x)} \ dx .\]
A formal second derivative operator $ -D^2 $ acts componentwise on functions $f \in L^2(\graph )$ in its domain $\domain$.  
Functions $f \in \domain$ are continuous on $\graph $, and continuously differentiable on each edge, with 
$f_e'$ absolutely continuous for each edge $e$, and $f'' \in L^2(\graph )$.
At each vertex $v$ the function $f$ must satisfy 
\begin{equation}\label{vcond1}
\sum_{e\sim v} \partial _{\nu} f_e(v) = 0,
\end{equation} 
where $e \sim v$ indicates an edge $e$ incident on the vertex $v$.
Here $\partial _{\nu}$ denotes the derivative computed in 'outward pointing' local coordinates which identify each edge $e$ incident on $v$
with a copy of $[0,1]$ and identify $v$ with $0$. The operator $\dop :L^2(\graph ) \to L^2(\graph ) $ acting by $-D^2$ with domain $\domain $ is self-adjoint \cite{BK}.
The spectrum is a subset of $[0,\infty )$.  

If the vertex set $\vset $ is finite then $\dop $ has discrete spectrum, with eigenvalues $\lambda _0 = 0 < \lambda _1 \le \lambda _2 \le \dots $,
listed with multiplicity, and an orthonormal basis of eigenfunctions $\phi _n$ with eigenvalues $\lambda _n$  \cite[p. 67]{BK}. 
The resolvent $R_{\graph} (\lambda ) = (\dop - \lambda I)^{-1} $ is trace class \cite[p. 206-212]{RS1} or \cite[p. 521-525]{Kato} 
with summable eigenvalue sequence $1/(\lambda _n - \lambda )$
for $\lambda $ in the resolvent set $\rho $, the complement of the spectrum. 
The resolvent can be written as an integral operator with kernel
\[R_{\graph}(\lambda ,x,y) = \sum_n \frac{\phi _n(x)\overline{\phi _n(y)} }{\lambda _n - \lambda } ,  \quad x,y \in \graph .\]
Since $\{ \phi _n \}$ is an orthonormal sequence the trace is 
\begin{equation} \label{trace1a}
{\rm tr} R_{\graph}(\lambda ) = \sum_n \frac{1}{\lambda _n - \lambda } = \int_{\graph } R(\lambda ,x,x) dx .
\end{equation}
The next proposition records some simple facts.

\begin{prop} \label{reality}
Suppose $\graph $ is finite.
If $\Im (\lambda ) \not= 0$ then $\Im (\lambda )\Im ({\rm tr}R_{\graph}(\lambda )) > 0$.  The function
${\rm tr}R_{\graph}(\lambda ) $ has exactly one root between distinct eigenvalues $\lambda _n$ and $\lambda _{n+1}$, and no other roots.
\end{prop}

\begin{proof}
Since $\lambda _n \in \real$,
\[\Im (\frac{1}{\lambda _n - \lambda }) = \frac{1}{2i} [  \frac{1}{\lambda _n - \lambda } -  \frac{1}{\lambda _n - \overline{\lambda }}]
=  \frac{1}{2i} \frac{\lambda - \overline{\lambda }}{|\lambda _n - \lambda |^2},\]
leading to the first claim, which also shows that any roots of ${\rm tr}R_{\graph}(\lambda ) $ must be real.  If $\lambda < 0$ 
then ${\rm tr}R_{\graph}(\lambda ) $ is a sum of positive terms, and there is a pole at $\lambda = 0$.
Finally, for $\lambda \in \real$, $\lim _{\lambda \downarrow \lambda _n} = -\infty$, $\lim _{\lambda \uparrow \lambda _n} = +\infty$
and
\[ \frac{d}{d\lambda } \sum_n \frac{1}{\lambda _n - \lambda } = \sum _n \frac{1}{(\lambda _n - \lambda )^2} ,\]
so ${\rm tr}R_{\graph}(\lambda ) $ is increasing between the eigenvalues.
\end{proof}

\subsection{Trace formula 1}

When $\graph $ is finite, a strong link between the eigenvalues of $\dop $ and a discrete Laplacian $\Delta $
leads to a formula for ${\rm tr}R_{\graph}(\lambda )$. 
Given a vertex $v \in \graph $, let $w_1,\dots ,w_{deg (v)}$ be the vertices adjacent to $v$.
On the vertex space $\vertspace $ consisting of functions $f:\vset \to \complex $, one has 
the adjacency operator  
\[Af(v) = \sum_{i=1}^{deg(v)} f(w_i),\]
and the degree operator  
\[Tf(v) = deg(v) f(v).\]
The operator $\Delta $ defined by 
\begin{equation} \label{Deltadef}
\Delta  f(v) = f(v) - T^{-1} Af(v) 
\end{equation}
is similar (in the sense of matrix conjugation)  \cite[pp. 3-7]{Chung} to the symmetric discrete Laplacian  $I - T^{-1/2} AT^{-1/2} $;
the eigenvalues are real and nonnegative.  In case $\graph $ is biregular the matrix $T^{-1/2} AT^{-1/2}$ is a constant multiple of $A$.

Suppose $y(x)$ is an eigenfunction of $\dop $ with eigenvalue $\lambda = \omega ^2$.
If the edge from $v$ to $w_i$ is identified with $[0,1]$, then 
\[y(x) = y(v)\cos (\omega x) + \frac{y(w_i) - y(v)\cos(\omega )}{\sin(\omega )} \sin (\omega x), \quad \sin(\omega ) \not= 0, \]
and $y$ may be recovered from its values at $0,1$ except when $ \sin(\omega ) = 0$.
The condition \eqref{vcond1}
gives
\[ f(v) \cos (\omega ) = \frac{1}{{\rm deg}(v)}\sum_j f(w_j) .\]
The argument can be reversed, leading to the following well known result \cite{Below, Cattaneo}
or \cite[p. 90-92]{BK}.

\begin{prop} \label{combspec}
If $\lambda _k \notin \{ n^2 \pi ^2 \ | \ n = 0,1,2,\dots \} $,
then $\lambda _k$ is an eigenvalue of $\dop $
if and only if $\nu  = 1 - \cos (\sqrt{\lambda _k} )$
is an eigenvalue of $\Delta $.  In this case $\lambda _k$ and $\nu $ have the same geometric multiplicity.
\end{prop}  

Henceforth, $\graph $ is assumed to be bipartite with $N_{\vset} $ vertices and $N_{\eset}$ edges.
Following \cite{CarlsonFFT},  additional spectral information is available.
Let $E(\lambda )$ denote the eigenspace of
$\dop $ with eigenvalue $\lambda $.
First a standard observation:
if $y \in E(0)$ then
\[ 0 = \int_{\graph} y\dop y = \int_{\graph} (y')^2,\]
so $y$ must be constant on each edge.   Since $y$ is continuous on $\graph $, $E(0)$ consists of the functions constant on $\graph $.

The edge space $\edgespace$ of $\graph $ consists of functions $f:\eset \to \complex$.
Pick a spanning tree $\tree $ for $\graph $, letting $e_1,\dots ,e_M$, with  $M = N_{\eset} - N_{\vset} + 1$, be the edges of $\graph $ not in $\tree $.
For each $m = 1,\dots ,M$ there is \cite[p. 53]{Bollobas} a cycle $C_m$ containing $e_m$ together with a path in $\tree $ connecting the vertices of $e_m$.
Let $Z_0(\graph )$ be the subspace of $\edgespace$ generated by the independent functions with the value $1$ on $C_m$ and zero on the complement. 

\begin{thm} \label{SpecA}  If $\graph $ is finite and bipartite, then
\[\dim E(n^2\pi ^2) = N_{\eset} - N_{\vset} + 2, \quad n = 1,2,\dots .\]
\end{thm}

\begin{proof} 
First note that if $y$ is an eigenfunction of $\dop $ with eigenvalue $n^2\pi ^2$ that  vanishes at any vertex, then it vanishes at all vertices.
Let $E_0(n^2\pi ^2) \subset E(n^2\pi ^2)$ denote those eigenfunctions of $\dop $ with eigenvalue $n^2\pi ^2$ 
vanishing at the vertices.

Since $\graph $ is bipartite all cycles $C_m$ are even.
Suppose $C_m$ has distinct vertices $v_0,\dots ,v_{2k-1}$ and edges $\{v_{j-1},v_{j}\}$ for $j  = 1,\dots ,2k-1$ and $\{ v_{2k-1},v_0\}$.
Pick coordinates $x$ identifying $\{v_{j-1},v_j\}$ with $[j-1,j]$ and $\{ v_{2k-1},v_0\}$ with $[2k-1,2k]$.
The functions $f_m$ with values $\sin(n\pi x)$ on the edges of $C_m$ and $0$ on all other edges are independent eigenfunctions of $\dop $. 

Now suppose that $y \in E_0(n^2\pi ^2)$. By subtracting a linear combination 
$ \sum a_mf_m $ we may assume that $y$ vanishes on the edges $e_m$. 
If $v$ is a boundary vertex $v$ of the spanning tree $\tree$ with $v$ in edge $e $ of  $\tree$, then $y = 0$ on all edges of $\graph $ incident on $v$ except possibly $e$.
But the vertex conditions \eqref{vcond1} force $y$ to vanish on $e$ as well, and thus on all of $\tree$.  
So ${\rm dim}E_0(n^2\pi ^2) = {\rm dim}Z_0(\graph ) = N_{\eset} - N_{\vset} + 1$.

Finally, identify each edge with $[0,1]$ so that the $R$ vertices are $0$ and the $B$ vertices are $1$. 
If $\lambda = n^2\pi ^2$, an eigenfunction vanishing at no vertices is given by $\cos (n \pi x)$ on each edge.
\end{proof}

\begin{thm} \label{SpecB} If $\graph $ is finite and bipartite, then the eigenvalues $\lambda $ of $\dop$ satisfy

(i) there are precisely $N_{\vset} - 2$ eigenvalues with $0 < \lambda < \pi ^2 $, counted with
multiplicity.  

(ii) there are precisely $N_{\vset} - 2$ eigenvalues with $\pi ^2 < \lambda < (2\pi )^2 $, counted with
multiplicity, and these have the form $\lambda = (2\pi - \mu )^2$ for $\mu $ an eigenvalue of $\dop $
with $0 < \mu < \pi ^2 $; $\lambda $ and $\mu $ have the same multiplicity. 

\end{thm}

\begin{proof} 

Since $\graph $ is bipartite \cite[p. 7]{Chung}, $\Delta $ has eigenvalues $\nu _0 = 0, \nu _{N_\vset -1} = 2$, both with multiplicity $1$.
All eigenvalues $\nu $ of $\Delta $ satisfy $0 \le \nu \le 2$.  Thus there are $N_{\vset} - 2 $ eigenvalues 
$\nu $ of $\Delta $ with $0 < \nu < 2$.  These correspond to eigenvalues $\lambda $ of $\dop $ with
$0 < \sqrt{\lambda } <  \pi $.  Since $1 - \cos (\sqrt{\lambda } )$ is symmetric about $\sqrt{\lambda} = \pi$, conclusion (ii) holds. 

\end{proof}

Let $\alpha _k^2$, $k = 1,\dots , 2N_{\vset} - 4$ be the eigenvalues of $\dop $, with multiplicity, satisfying $0 < \alpha _k^2 < (2\pi )^2$
and $\alpha _k^2 \not= \pi ^2$.  The description of the eigenvalues of $\dop $ implies that the trace of $R_{\graph }(\lambda )$ can be rewritten as
\[ {\rm tr}(R_{\graph}(\lambda )) = - \frac{1}{\lambda} + \sum_{k=1}^{2N_{\vset} - 4} \sum _{m = -\infty}^{\infty}  \frac{1}{ [\alpha _k + 2m\pi ]^2 - \lambda }
+ (N_{\eset} - N_{\vset} + 2)\sum_{m=1}^{\infty}  \frac{1}{ [m\pi ]^2 - \lambda }\]

A standard residue calculation will be sketched to evaluate these sums; more details are in \cite[p.149]{Marshall}. 
The function $\pi \cot (\pi z)/ [(\alpha + 2\pi z)^2 - \lambda ]$ has simple poles with residues $1/ [(\alpha + 2\pi m)^2 - \lambda ]$
at the integers $m$, and simple poles at $z = (- \alpha \pm \sqrt{\lambda } )/(2\pi )$.  Integrate over the contours $\gamma _N$
which bound the squares with vertices $(N+1/2) (\pm 1 \pm i)$ to get
\[0 = \lim _{N \to \infty} \frac{1}{2\pi i} \int_{\gamma _N}  \frac{\pi \cot (\pi z)}{(\alpha + 2\pi z)^2 - \lambda } \ dz
= \sum_{m = -\infty}^{\infty} \frac{1}{ (\alpha + 2\pi m)^2 - \lambda } \]
\[+ \frac{1 }{4\sqrt{\lambda } } \cot (\frac{-\alpha + \sqrt{\lambda}}{2})
- \frac{1 }{4\sqrt{\lambda } } \cot (\frac{-\alpha - \sqrt{\lambda}}{2}). \]
Trigonometric identies give
\begin{equation} \label{res1}
\sum_{m=-\infty}^{\infty} \frac{1}{[\alpha + 2\pi m  ]^2 -\lambda  }
 =  \frac{\sin(\sqrt{\lambda  })}{2 \sqrt{\lambda  }}
\frac{ 1} { \cos(\sqrt{\lambda  }) - \cos(\alpha ) } 
\end{equation}

Similarly,
\[\sum_{m=1}^{\infty}  \frac{1}{ [m\pi ]^2 - \lambda } = \frac{1}{2\lambda } - \frac{\cot(\sqrt{\lambda })}{2\sqrt{\lambda }},\]
and so
\begin{equation} \label{trace1b}
 {\rm tr}(R_{\graph} (\lambda )) =  \frac{N_{\eset} - N_{\vset} }{2\lambda} + \frac{\sin(\sqrt{\lambda })}{2 \sqrt{\lambda }}
 \sum_{k=1}^{2N_{\vset} - 4} \frac{1}{\cos (\sqrt{\lambda }) - \cos (\alpha _k)} 
 \end{equation}
\[ - (N_{\eset} - N_{\vset} + 2)\frac{\cot(\sqrt{\lambda })}{2\sqrt{\lambda }}\]

\section{Biregular graphs}

A graph $\graph $ is bipartite if the vertex set $\vset$ can be partitioned into two subsets, $\vset _R$ and $\vset _B$, and 
all edges $e$ of the edgeset $\eset$ have the form $e = \{ r,b \}$ with $r \in \vset _R$ and $b \in \vset _B$.  
$\graph $ is biregular if the degree of each vertex depends only on the vertex class, $R$ or $B$. 
A simple example is the complete bipartite graph.   Another important example is the biregular tree $\tree$,
with vertices in $\vset _B$ having degree $ \delta _B +1 > 1$, and vertices in $\vset _R$ have degree $ \delta _R +1 > 1$.

Here are two ways of generating biregular graphs from a regular graph $\Gamma $.  Suppose $\vset _1$ and $\vset _2$ are two copies
of the vertex set of $\Gamma $.  The bipartite cover $\widetilde \Gamma $ of $\Gamma $ has vertex set the disjoint union of $\vset _1$ and $\vset _2$.
If $v_1 \in \vset _1$ and $v_2 \in \vset _2$, then there is an edge $\{v_1,v_2 \} $ in $\widetilde \Gamma $ if and only if the corresponding vertices
$v_1,v_2$ in $\Gamma $ are adjacent.  $\widetilde \Gamma $ is connected if and only if $\Gamma $ is connected and not bipartite.
The second method works well for quantum graphs.  For each edge $e$ of $\Gamma $, simply add a vertex $v$ to the middle of $e$,
using the continuity and derivative condition of \eqref{vcond1}.  The vertex $v$ will have degree $2$, but the domain of $\dop $ is not affected
\cite[p. 14]{BK}.  When all edges are required to have length $1$, the edges of the original graph effectively have length $2$.
Rescaling of edge lengths for a finite graph by a factor $2$ converts eigenvalues of $\dop $ from $\lambda _n$ to $\lambda _n/4$. 
In this fashion the spectral theory of regular quantum graphs is subsumed into the 
spectral theory of biregular quantum graphs.

Computations for the complete bipartite graphs will be useful later.  The complete bipartite graph $K(m_B,m_R)$ has $m_B$ vertices in $\vset _B$ and
$m_R$ vertices in $\vset _R$, with edges $\{ v,w \}$ whenever $v \in \vset _B$ and $w \in \vset _R$.   To compute the $\Delta $ spectrum for $K(m_B,m_R)$, 
index the $R$ vertices as $v_1,\dots , v_{m_R}$ and the $B$ vertices as $v_{m_R+1} , \dots , v_{m_B+m_R}$.
The matrix $T^{-1}A$ of \eqref{Deltadef} is 
\[T^{-1}A = \begin{pmatrix} 0_{m_R,m_R} & C_1 \cr C_2 & 0_{m_B,m_B}, 
\end{pmatrix} \]
where $C_1$ is an $m_R \times m_B$ block with all entries $1/m_B$ and $C_2$ is an $m_B \times m_R$ block with all entries $1/m_R$.
Since 
\[(T^{-1}A)^2 =  \begin{pmatrix} C_1C_2 & 0_{m_R,m_B} \cr  0_{m_B,m_R} & C_2C_1 
\end{pmatrix}, \]
and ${\rm tr}T^{-1}A = 0$ one finds that 
$0$ is an eigenvalue of $T^{-1}A$ with multiplicity $m_R+m_B-2$, while the other eigenvalues are
$\pm 1$.
Thus the eigenvalues of $\Delta = I - T^{-1} A$ are $1$ with multiplicity $m_R+m_B-2$, $0$ and $2$.

By \propref{combspec} the eigenvalues $\lambda \not= n^2\pi^2 $ of the corresponding differential operator $\dop $ are
$(\frac{\pi}{2} + n\pi )^2$ for $n=0,1,2,\dots$, each with multiplicity $m_R+m_B-2$.  By \thmref{SpecA} the multiplicity of $\lambda = n^2\pi ^2$
for $n = 1,2,3,\dots $ is $m_Rm_B - (m_R+m_B) + 2$.  Finally $\lambda = 0$ is a eigenvalue with multiplicity $1$.
In particular, $\cos (\alpha _k) = 0$.  Letting  ${\rm tr}R_{CB}(\lambda )$ denote the trace of the resolvent for $\dop $ on a complete bipartite graph, 
\eqref{trace1b} gives
 \begin{equation} \label{comptrace}
 {\rm tr}R_{CB}(\lambda ) =  \frac{m_Rm_B - (m_R+m_B) }{2\lambda} +  (\delta_R+ \delta _B )\frac{\tan(\sqrt{\lambda })}{\sqrt{\lambda }} 
  \end{equation}
\[  - (m_Rm_B - (m_R+m_B) + 2 )\frac{\cot(\sqrt{\lambda })}{2\sqrt{\lambda }}.\]

\section{Biregular trees}

\subsection {Multipliers} 

Now consider the case of the biregular tree $\tree$
whose edges all have length $1$.  Vertices in $\vset _B$ have degree $d_B = \delta _B +1 > 1$,
and vertices in $\vset _R$ have degree $d_R = \delta _R +1 > 1$,

Fix a vertex $r_0$, identified with $0 \in \real$, and an initial edge $e_0 = \{ r_0,b_0 \}$, which is identified with the interval  $[0,1]$. 
Two classes of rays, extending from the edge $e_0$, will be used: one whose vertex sequence starts $r_0,b_0$, the other starting with $b_0,r_0$.
The rays rooted at $r_0$ are subgraphs of $\tree$ with vertex set $\{ r_0,b_0, r_1,b_1,\dots \}$,  indexed by nonnegative integers $n$, with edges $\{ r_n,b_n \}$ and $\{ b_n, r_{n+1} \}$.  
The rays rooted at $b_0$ will have a vertex sets $\{ b_0,r_0,b_{-1},r_{-1}\dots \}$ indexed by nonpositive integers $n$, 
with edges $\{ b_n,r_n \}$ and $\{r_{n}, b_{n-1} \}$ .
As metric graphs, a ray rooted at $r_0$ may be identified with $[0,\infty )$, and a ray rooted at
$b_0$ may be identified with $(-\infty , 1]$.
With these identifications, even integers $2n$ correspond to vertices $r_n$ and odd integers $2n+1$ correspond to vertices $b_n$.  Let $\tree _0$ (resp. $\tree _1$) denote the subtree of $\tree $ obtained as the union of the rays
rooted at $r_0$ (resp. $b_0$).  

Solutions $y(x,\lambda )$ of 
\begin{equation} \label{1.a}
-D^2y = \lambda y
\end{equation}
on $[0,1] $ will be extended to $\tree _0$ (resp. $\tree _1$), 
with values depending only on $x \in [0,\infty )$ (resp. $x \in (-\infty , 1]$).
These extended solutions will be continuous at each vertex, with \eqref{vcond1} satisfied.
Such solutions $y(x,\lambda )$ 
may be continued from edge $e_0 $ along the indicated rays by 
using the following jump conditions at integers $x$. The continuity requirement is $y(x^-,\lambda ) = y(x^+,\lambda )$.
For rays rooted at $r_0$ the derivative conditions are satisfied by taking 
\begin{equation} \label{jumps0}
\begin{matrix} y'(x^+,\lambda ) = y'(x^-,\lambda )/\delta _R, & x > 0, {\rm even}, \cr
y'(x^+,\lambda ) = y'(x^-,\lambda )/\delta _B, & x > 0, {\rm odd}, 
\end{matrix}
\end{equation}
while for rays rooted at $b_0$ the conditions are 
\begin{equation} \label{jumps1}
\begin{matrix} y'(x^-,\lambda ) = y'(x^+,\lambda ) /\delta _R, & x < 1, {\rm even}, \cr
y'(x^-,\lambda ) = y'(x^+,\lambda )/\delta _B, & x < 1, {\rm odd}.
\end{matrix}
\end{equation}
The propagation of initial data $y(v,\lambda ),y'(v,\lambda )$ for the extended solutions is described by transition matrices. 
When initial data is decomposed using the eigenvectors of the transition matrix, the eigenvalues act as multipliers.  

Let $\omega = \sqrt{\lambda }$, with $\omega \ge 0$ for $\lambda \ge 0$ and the square root taken continuously 
for $\lambda \in \complex \setminus i (-\infty , 0)$.
On the initial edge $e_0$ the equation \eqref{1.a} has a solution basis  
$\cos (\omega x), \sin (\omega x)/\omega $.  At $x = 0^+$ this basis satisfies
\begin{equation}
\begin{pmatrix} \cos (\omega 0^+) & \sin (\omega 0^+)/\omega \cr -\omega \sin (\omega 0^+) & \cos (\omega 0^+) \end{pmatrix}
= \begin{pmatrix} 1 & 0 \cr 0&1 \end{pmatrix}.
\end{equation}
The abbreviations 
\[c(\lambda ) = \cos (\omega ), \quad  c'(\lambda ) = -\omega \sin (\omega ) \]
\[s(\lambda ) = \sin (\omega )/\omega , \quad s'(\lambda ) = \cos (\omega ) \]
will be used.

Suppose a solution $y$ of (\ref{1.a}) satisfies $y(0^+,\lambda  ) = a$ and $y'(0^+, \lambda  ) = b$.
At $x = 1^-$ the solution will have values given by
$$\begin{pmatrix} y(1^-,\lambda  ) \cr y'(1^-,\lambda  )\end{pmatrix} 
= M_0(\lambda  )\begin{pmatrix}a \cr b\end{pmatrix}, \quad 
 M_0(\lambda  ) = \begin{pmatrix} c(\lambda ) & s(\lambda ) \cr
c'(\lambda ) & s'(\lambda ) \end{pmatrix}.$$
By \eqref{jumps0}, the initial data at $x =1^+$ is then 
\[\begin{pmatrix} y(1^+,\lambda  ) \cr y'(1^+,\lambda  )\end{pmatrix} = J_{0,B}\begin{pmatrix} y(1^-,\lambda  ) \cr y'(1^-,\lambda  )\end{pmatrix}, \quad  
J_{0,B} =  \begin{pmatrix}1 & 0 \cr 0  & 1/\delta _B \end{pmatrix}.\]
Initial data propagation from $x=1^+$ to $x=2^-$ is again given by multiplication by $M_0(\lambda )$, while  
\[\begin{pmatrix} y(2^+,\lambda  ) \cr y'(2^+,\lambda  )\end{pmatrix} = J_{0,R}\begin{pmatrix} y(2^-,\lambda  ) \cr y'(2^-,\lambda  )\end{pmatrix}, \quad  J_{0,R} =  \begin{pmatrix}1 & 0 \cr 0 & 1/\delta _R \end{pmatrix}.\]
Continuing in this fashion, the propagation of initial data from $x = 2n^+$ to $x = (2n+2)^+$ for $n \ge 0$, is given by
\[ T_0(\lambda ) = J_{0,R}M_0J_{0,B}M_0 \]
\[= \begin{pmatrix}1 & 0 \cr 0 & 1/\delta _R \end{pmatrix}
\begin{pmatrix} c(\lambda ) & s(\lambda ) \cr
c'(\lambda ) & s'(\lambda ) \end{pmatrix}  \begin{pmatrix}1 & 0 \cr 0  & 1/\delta _B \end{pmatrix}
\begin{pmatrix} c(\lambda ) & s(\lambda ) \cr
c'(\lambda ) & s'(\lambda ) \end{pmatrix} \]
\[= \begin{pmatrix} c^2 + sc' /\delta _B  &  cs + ss'/ \delta _B 
 \cr
 c'c/\delta _R + s'c'/(\delta _R\delta _B) & sc'/\delta _R  
+  (s')^2/(\delta _R\delta _B)
\end{pmatrix} \]
Note that 
\begin{equation} \label{trace0}
\det (T_0(\lambda ) ) = \det (J_{0,R}M_0J_{0,B}M_0) = \frac{1}{\delta _B \delta _R}, \end{equation}
\[trT_0(\lambda )  
= \cos ^2(\omega )(1 + \frac{1}{\delta _B\delta _R}) - \sin ^2(\omega ) (\frac{1}{\delta _R} + \frac{1}{\delta _B} ).\]

Next, treat the behavior of similar solutions which start at $x=1^-$ and extend to decreasing values of $x$.
Still using the basis $\cos (\omega x), \sin (\omega x)/\omega $, the matrix taking initial data from $x=1^-$ to $x=0^+$ is
\[M_1(\lambda  ) = M_0^{-1}(\lambda  ) = 
\begin{pmatrix} s'(\lambda ) & -s(\lambda ) 
\cr -c'(\lambda ) & c(\lambda ) \end{pmatrix}.\]
From \eqref{jumps1} the jump in initial data from $x = 0^+$ to $x=0^-$ is given by 
\[J_{1,R} =  \begin{pmatrix}1 & 0 \cr 0 & 1/\delta _R \end{pmatrix}.\]
Similarly, the jump from $x = -1^+$ to $x=-1^-$ is given by
\[J_{1,B} =  \begin{pmatrix}1 & 0 \cr 0  & 1/\delta _B \end{pmatrix}.\]
More generally, the propagation of initial data from $x = (2n+1)^-$ to $x = (2n-1)^-$ for $n \le 0$, is given by
\[ T_1(\lambda ) = J_{1,B}M_1J_{1,R}M_1 \]
\[= \begin{pmatrix}1 & 0 \cr 0  & 1/\delta _B \end{pmatrix}
\begin{pmatrix} s'(\lambda ) & -s(\lambda ) \cr
-c'(\lambda ) & c(\lambda ) \end{pmatrix}  
\begin{pmatrix}1 & 0 \cr 0 & 1/\delta _R \end{pmatrix}
\begin{pmatrix} s'(\lambda ) & -s(\lambda ) \cr
-c'(\lambda ) & c(\lambda ) \end{pmatrix} \]
\[= \begin{pmatrix} (s')^2 + sc'/ \delta _R & -s's - sc/ \delta _R  \cr
-c's'/\delta _B - cc'/(\delta _B\delta _R) & c's /\delta _B 
+  c^2/(\delta _B\delta _R) 
\end{pmatrix} \]

Again $\det (T_1(\lambda )) = \frac{1}{\delta _B \delta _R}$ and ${\rm tr}(T_1(\lambda )) = {\rm tr}(T_0(\lambda ))$
Since the transition matrices $T_j(\lambda )$ have the same determinants and traces,
they have the same eigenvalues 
\begin{equation} \label{quadform} 
\mu ^{\pm}(\lambda ) = {\rm tr} T_j/2 \pm \sqrt{{\rm tr}(T_j)^2/4 - \det(T_j)},
\end{equation}
Again, take the square root nonnegative when the argument is nonnegative, and continuous 
if the argument is in $\complex \setminus i (-\infty , 0)$.  

Suppose the $2 \times 2$ matrix 
\[ \begin{pmatrix} a & b \cr c & d \end{pmatrix} \]
has eigenvalues $\mu ^{\pm}$.  Then
\[  \begin{pmatrix} a - \mu ^{\pm} & b \cr c & d -\mu ^{\pm} \end{pmatrix} \begin{pmatrix} b \cr \mu ^{\pm} - a \end{pmatrix} = \begin{pmatrix} 0 \cr 0 \end{pmatrix},\]
so the vectors $[b, \mu ^{\pm} - a]$ are eigenvectors.
Using $s'(\lambda ) = c(\lambda )$ these eigenvectors for $T_j(\lambda )$ are  
\begin{equation} \label{evecs}  
E_0^{\pm}(\lambda ) = \begin{pmatrix} sc + sc/\delta _B
\cr \mu ^{\pm} - c^2 - sc'/\delta _B \end{pmatrix}, \quad  E_1^{\pm}(\lambda ) =  \begin{pmatrix} -sc - sc/\delta _R
\cr \mu ^{\pm} - c^2 - sc'/ \delta _R \end{pmatrix} 
\end{equation}
The eigenvectors $E_0^{\pm}(\lambda ), E_1^{\pm}(\lambda )$ are nonzero except for a discrete set of $\lambda \in \real $.

When $\lambda  $ is real the matrices $T_j(\lambda  )$ are real-valued.   Then
if ${\rm tr}(T_j)^2/4 - \det (T_j)$ is nonpositive the eigenvalues are conjugate pairs, with
\begin{equation} \label{eval1} |\mu ^{\pm}(\lambda ) | = \frac{1}{\sqrt{\delta _B \delta _R}}, \quad \frac{- 2}{\sqrt{\delta _B\delta _R }} \le {\rm tr}(T_j) \le \frac{2}{\sqrt{\delta _B\delta _R} }.
\end{equation}
If ${\rm tr}(T_j)^2/4 - \det (T_j)$ is nonnegative, the eigenvalues are real.
Since $\mu ^+ \mu ^- = 1/(\delta _B\delta _R ) $, the eigenvalues  
have the same sign.

\begin{lem} \label{asyneg}
The eigenvalues $\mu ^{\pm} (\lambda )$ satisfy  
\begin{equation} \label{muest}
\lim_{\lambda \to -\infty } e^{-2|{\rm Im} \sqrt{\lambda }|} \mu ^+(\lambda ) =  \frac{\delta _R +1}{2\delta _R}  \frac{\delta _B + 1}{2\delta _B}, 
\quad  \lim_{\lambda \to -\infty } e^{2|{\rm Im}\sqrt{\lambda }|} \mu ^-(\lambda ) = \frac{2}{\delta _R + 1}  \frac{2}{\delta _B + 1}
\end{equation}
\end{lem}

\begin{proof}

For $\lambda < 0$ let $\omega = \sqrt{\lambda }$ as before.  In exponential form
\[ c(\lambda ) = (e^{i\omega } + e^{-i \omega})/2 =  (e^{-|\omega |} + e^{|\omega |})/2
, \quad c'(\lambda ) = - \omega (e^{i\omega } - e^{-i\omega})/2i,\]
\[s(\lambda ) = (e^{i\omega  } - e^{-i \omega })/(2i\omega ) = (e^{-|\omega |} - e^{|\omega |})/(2i\omega ), \]

Since $ {\rm tr} T_0(\lambda ) =  {\rm tr} T_1(\lambda )$, equation \eqref{trace0} gives  
\[ \lim_{\lambda \to -\infty} e^{-2|{\rm Im} \sqrt{\lambda }|} {\rm tr}\,(T_j)  = \frac{1}{4} [
1 + \frac{1}{\delta _R}  + \frac{1}{\delta _B} + \frac{1}{\delta _R\delta _B}] = \frac{\delta _R +1}{2\delta _R}  \frac{\delta _B + 1}{2\delta _B}\]
\eqref{quadform} first gives the estimate for $\mu ^+(\lambda )$, with the estimate for $\mu ^-(\lambda )$ following from 
$\mu ^+\mu ^- = \frac{1}{\delta _R\delta _B}$.
\end{proof} 

\begin{lem}\label{cross} \it If $ | \mu ^{\pm}(\lambda  ) | = 1/\sqrt{\delta _B\delta _R}$, 
then $\lambda  $ is in the spectrum of ${\cal L}$, and so is real.\end{lem}

\begin{proof} The cases being similar, suppose that $y$ is a nontrivial 
solution of (\ref{1.a}) on $e_0$ whose initial data at the left endpoint $x = 0$ is an 
eigenvector for $T_0(\lambda  )$ with eigenvalue $\mu ^- $.
Extend $y$ to the subtree $\tree _0$ using $y(x+2) = \mu ^- y(x)$ along each ray based at $r_0$.
The self adjoint conditions \eqref{vcond1} hold for $y$ except possibly at $r_0$.

These extended solutions decay as $x \to \infty $.  For $2(n+1) > 0$ let $\tree _{2n+1}$ denote the subtree of
$\tree _0$ with $x \le (2n+1)$.  Then 
\[\int_{\tree _{2n+1}} | y |^2 \ge \sum_{k=1}^{n} (\delta _B \delta _R)^k \int_{0 \le x \le 2} |(\mu ^-)^k y |^2
= (n-1) \int_{0 \le x \le 2} |y |^2,\]
while $y$ satisfies \eqref{1.a} on $\tree _{2n+1}$.  The functions $y$ on $\tree _{2n+1}$ can be modified \cite[Lem 3.2]{Carlson97}
on the edges with $0 \le x \le 1$ and $2n \le x \le 2n+2$ to put $y$ in the domain of $\dop $ with $\| \dop y \| _2 \le C$.
Then $f = y/\|y\| _2 $ has norm $1$ while  $\| \dop f \| _2 \to 0$ as $n \to \infty $.   
$\lambda  $ is thus in the spectrum of the self adjoint operator $\dop $. 
\end{proof}

By \lemref{cross} the set 
\[ \sigma _1 = \{ \lambda  \in \complex \ \bigl | \ |\mu^{\pm}(\lambda  )| = |\delta _R\delta _B |^{-1/2} \}
= \{ \lambda \in \real \ \big | {\rm tr}(T_j)^2 - 4\det(T_j) \le 0 \} \]
is contained in a half line $[0,\infty )$.
The next result describes the main features of $\mu^{\pm}(\lambda )$.

\begin{thm}\label{dom1} The eigenvalues $\mu^{\pm}(\lambda  )$ may be chosen to be 
single valued analytic functions in the complement of $\sigma _1$ with the asymptotic 
behaviour given in \lemref{asyneg}.  On this domain $|\mu ^+(\lambda  )| > 1/\sqrt{\delta _R\delta _B}$
and $|\mu ^-(\lambda  )| < 1/\sqrt{\delta _R\delta _B}$.
These functions $\mu ^+(\lambda )$ and $\mu ^-(\lambda )$ have continuous extensions to the real axis 
from the upper and lower half planes which are analytic except on the discrete set where ${\rm tr}(T_j)^2 = 4/(\delta _R \delta _B)$.
If $\nu \in \sigma _1$ then
\begin{equation}
\lim_{\epsilon \to 0^+} \mu^{\pm}(\nu + i \epsilon )
- \mu^{\pm}(\nu - i \epsilon ) = 2i\,{\rm Im}\,(\mu^{\pm}(\nu ))\,. 
\label{3.g} \end{equation}
\end{thm}

\begin{proof}  Since ${\rm tr}(T_j)$ and $\det (T_j)$ are entire functions of $\lambda  $, 
the eigenvalues $\mu ^{\pm }(\lambda )$ (and eigenvectors \eqref{evecs}) will be analytic in any simply 
connected domain with ${\rm tr}(T_j)^2 - 4/(\delta _R \delta _B) \not= 0$.  Use $U = \complex \setminus [0,\infty )$ as such a domain.

The eigenvalues satisfy $\mu ^+\mu ^- = 1/(\delta _R\delta _B) $, so $|\mu ^+|$ and $|\mu ^-|$ and are distinct in $\complex \setminus \sigma _1$. 
Since $|\mu ^+(\lambda  )| > 1/\sqrt{\delta _R\delta _B }$ and $|\mu ^-(\lambda  )| < 1/\sqrt{\delta _R\delta _B}$ 
in $U$, the extensions across $\real \setminus \sigma _1$ are single valued.

To obtain the continuous extension to the real axis, note that 
the set of points where the eigenvalues coalesce, or ${\rm tr}(T_j)^2 - 4/(\delta _R \delta _B) = 0$,
is the zero set of an entire function, which has isolated (real) zeroes $r_i$.
The analytic functions $\mu^{\pm}(\lambda )$, given by \eqref{quadform}, thus have an analytic continuation from either half
plane to the real axis with these points $r_i$ omitted.  At these points
$$\lim_{\lambda  \to r_i} \mu^{\pm}(\lambda  ) = \mbox{tr}\,(T_j)/2$$
independent of the branch of the square root. 

We have observed in \lemref{cross} that if $|\mu^{\pm}(\nu )| = 
1/\sqrt{\delta _R\delta _B }$, then
$\nu \in \real $.  If ${\rm tr}\,(T_j)^2= \det (T_j) = 4/(\delta _R\delta _B) $,
then both sides of (\ref{3.g}) are $0$.  Suppose instead that 
${\rm tr}\,(T_j)^2/4 - \det (T_j) <0$,
so that the eigenvalues $\mu^{\pm}(\nu )$ are a nonreal conjugate pair.
Since the eigenvalues are distinct, they extend analytically across the real axis.
There are two possibilities: either (i) (\ref{3.g}) holds, in which case the extension of
$\mu^{\pm}(\nu +i\epsilon )$ is $\mu^{\mp}(\nu - i\epsilon )$, 
or (ii) $\mu^{\pm}(\nu + i\epsilon )$ extends to $\mu^{\pm}(\nu - i\epsilon )$.
The second case will be excluded because $|\mu ^+| > 1/\sqrt{\delta _R\delta _B}$
in the complement of $\sigma _1 $.
If (ii) held, then $\mu ^+$ would be an analytic function of $\lambda  $ in a neighborhood of $\nu$ 
satisfying $|\mu^+(\nu)| = 1/\sqrt{\delta _R\delta _B}$, and  
$|\mu^+(\lambda  )| \ge 1/\sqrt{\delta _R\delta _B}$.  But this violates the open mapping theorem
\cite[p. 162]{GK}, so (i) must hold. The treatment of $\mu ^-$ is the same.
\end{proof}

Define additional solutions of (\ref{1.a}) on $(0,1)$ by
\[ V(x,\lambda )  = sc(1+ \frac{1}{\delta _B} )\cos(\omega x)   
+ [ \mu ^{-} - c^2 -  \frac{sc'}{\delta _B} ] \sin(\omega x)/\omega ,\]
\[ U(x,\lambda ) 
= -sc( 1 + \frac{1}{\delta _R} )\cos (\omega (1-x))  
- [ \mu ^{-} - c^2 -  \frac{sc'}{\delta _R}]\sin(\omega (1-x))/\omega .\]
$V(x,\lambda )$ (resp. $U(x,\lambda )$) has initial data at $x=0^+$ (resp. $x=1^-$) given by the eigenvector $E_0^-$ (resp. $E_1^-$).

\begin{thm}\label{UandV} \it If $\lambda  \in \complex \setminus \sigma _1$ then $V(x,\lambda )$ may be extended 
to a solution of (\ref{1.a}) on $\tree _0$ which satisfies the vertex conditions \eqref{jumps0}, and so \eqref{vcond1}, for $x > 0$, 
and is square integrable on $\tree _0$.  
The analogous statements hold for $U(x,\lambda) $ on $\tree _1$.

\end{thm}

\begin{proof} If $\lambda  \in \complex \setminus \sigma _1$ then    
by \thmref{dom1} the eigenvalue $\mu ^-(\lambda  )$ is well defined.
Extend $V(x,\lambda )$ as a solution to (\ref{1.a}) on $x > 1$ using the conditions \eqref{jumps0}. 
For integers $n > 0$, the initial data at $x=2n^+$ is given by 
\[\begin{pmatrix} V(2n^+,\lambda ) \cr V'(2n^+,\lambda ) \end{pmatrix} = 
T_0^n(\lambda ) \begin{pmatrix} V(0^+,\lambda ) \cr V'(0^+,\lambda ) \end{pmatrix}
= (\mu ^-)^n \begin{pmatrix} V(0^+,\lambda ) \cr V'(0^+,\lambda ) \end{pmatrix} .\]
Since $|\mu ^-(\lambda  )| < 1/\sqrt{\delta _R\delta _B}$ the computation
\[\int_{\cal T _R} | y |^2 = \sum_{k=0}^{\infty} (\delta _B \delta _R)^k \int_{0 \le x \le 2} |(\mu ^-)^k y |^2
= \sum_{k=0}^{\infty} \big ( \delta _B \delta _R (\mu ^-)^2 \big )^k \int_{0 \le x \le 2} |y |^2,\]
establishes the square integrability.
\end{proof}

\subsection{The resolvent and spectrum for trees}  

The resolvent $(\dop - \lambda I)^{-1}$ is defined in the resolvent set $\rho $, the complement of the spectrum of $\dop $. 
The solutions $U$ and $V$ of \thmref{UandV} can be used to construct the
resolvent of ${\cal L}$ on ${\cal T}$. The Wronskian $W(\lambda ) = W(U,V) = U(x,\lambda )V'(x,\lambda )-U'(x,\lambda )V(x,\lambda )$,
defined for $x \in (0,1)$, is independent of $x$.   Evaluation at $x = 1^-$ gives
\begin{equation} \label{Wronski}
W(\lambda  ) = 
\Bigl [ sc(1 + \frac{1}{\delta _B})c' + ( \mu ^{-} - c^2 -  \frac{sc'}{\delta _B} )c \Bigr ] sc(1 + \frac{1}{\delta _R })
\end{equation}
\[+ s\Bigl [(1 +  \frac{1}{\delta _B} )c^2  + ( \mu ^{-} - c^2 -  \frac{sc'}{\delta _B} ) \Bigr ] \Bigl [( \mu ^{-} - c^2 -    \frac{sc'}{\delta _R} ) \Bigr ] \]
The identity $ sc' - c^2 = sc' - cs' = -1$ leads to a simplification,
\[W(\lambda ) =  s \Bigl [ c^2(1 + \frac{1}{\delta _B}) +  ( \mu ^{-} - c^2 -  \frac{sc'}{\delta _B} ) \Bigr ]
\Bigl [  c^2(1 + \frac{1}{\delta _R}) +  ( \mu ^{-} - c^2 -  \frac{sc'}{\delta _R} ) \Bigr ] \]
\[-  sc^2(1 +  \frac{1}{\delta _B} ) ( 1 + \frac{1}{\delta _R} ) \]
\[ =  s \Bigl [ \frac{1}{\delta _B}  +   \mu ^{-}  \Bigr ]
\Bigl [   \frac{1}{\delta _R} +   \mu ^{-}  \Bigr ] 
-  sc^2(1 +  \frac{1}{\delta _B} ) ( 1 + \frac{1}{\delta _R} ).\]

Define
\[\sigma _2 = \{ \lambda \in \complex \setminus \sigma _1 \ | \ W(\lambda ) = 0 \} .\]

\begin{lem} \label{Wron}  
$\sigma _2$ is a discrete subset of $\real \setminus \sigma _1$.
If $\delta _R\delta _B > 1$ and $s(\lambda ) = 0$ then $\lambda \in \sigma _2$.
If $\lambda \in \sigma _2$ and  
$U$ and $V$ are not identically zero, which is the case if  $s(\lambda ) = 0$,
then $\lambda $ is a eigenvalue of $\dop $.   
\end{lem}

\begin{proof}
Having initial data $E_0^-(\lambda )$ and $E_1^-(\lambda )$, the solutions $V(x,\lambda )$ and $U(x,\lambda )$ are not identically $0$ 
except for a discrete set of $\lambda \in \real $. 
The vanishing of $W(\lambda )$ means that $V$ and $U$ are linearly dependent on $(0,1)$.  If $U(x,\lambda ) = bV(x,\lambda )$ 
for some constant $b$ and $\lambda \in \complex \setminus \sigma _1$, the function 
\[w(x,\lambda ) =  \Bigl \{ \begin{matrix} bV(x,\lambda ), x \in \tree _0 \cr U(x,\lambda ) , x \in \tree _1 \end{matrix} \Bigr \} \]
is an nontrivial eigenfunction of $\dop $.  Since $\dop $ is self-adjoint, $\lambda \in \real $ and $W(\lambda ) \not= 0$ for $\lambda \in \complex \setminus \real$.  The discreteness of $\sigma _2$ follows since $W(\lambda )$ is analytic in $\complex \setminus \sigma _1$. 

If $s(\lambda ) = 0$ then $W(\lambda ) = 0$ and $c^2(\lambda ) = 1$.
Thus $tr(T_j(\lambda ))^2/4 - \det T_j(\lambda ) = (1 - 1/(\delta _B\delta _R))/4$, which is strictly positive if $\delta _R\delta _B > 1$, and $\lambda \notin \sigma _1$ by \eqref{quadform}. Since $|\mu ^-| < 1/\sqrt{\delta _B\delta _R}$, the initial data for $V(x,\lambda )$ and $U(x,\lambda )$ is nonzero by \eqref{evecs}.

\end{proof}

For $\lambda  \in \complex \setminus (\sigma _1 \cup \sigma _2)$ define the kernel
\begin{equation} \label{resform}
R_e(x,t,\lambda  ) = 
\left\{\begin{array}{ll}
-U(x,\lambda  )V(t,\lambda  )/W, \quad  -\infty < x \le t, 0 \le t \le 1\,, \\ 
-U(t, \lambda  )V(x, \lambda  )/W, \quad 0 \le t \le 1, t \le x < \infty \,.  
\end{array}\right\}. 
\end{equation}
Suppose $f_e:[0,1] \to \complex$ has support in $(0,1)$.  A simple computation \cite [p. 309]{BR} shows that the function 
\[h_e(x) = \int_0^1 R_e(x,t,\lambda  )f_e(t) \, dt \]
satisfies $[-D^2-\lambda ]h_e = f_e$.  By \thmref{UandV} $h_e$ is square integrable on ${\cal T}$ and satisfies the boundary conditions \eqref{vcond1}
at each vertex of $\tree $.  Thus $h_e$ is in the domain of ${\cal L}$.
Since integration of $f_e$ against the meromorphic kernel $R_e(x,t,\lambda  )$ 
agrees with the resolvent $R(\lambda  )f_e$ as long as $W(\lambda  ) \not= 0$, they must agree for all $\lambda  \in \rho $.

For $f \in L^2({\cal T})$, let $f_e$ denote the restriction of $f$ to the edge $e$ with the edge orientation described for \eqref{resform}.  

\begin{thm}\label{Thm4.1} For $\lambda  \in  \complex \setminus (\sigma _1 \cup \sigma _2)$,
\begin{equation} \label{resform2}
R(\lambda  )f = \sum_e \int_0^1 R_e(x,t,\lambda  )f_e(t) \, dt ,
\end{equation} 
the sum converging in $L^2({\cal T})$. \end{thm}

\begin{proof} The formula \eqref{resform2} agrees with the resolvent of $\dop $ in the resolvent set, which includes $\complex \setminus \real $.
Let $[a,b]$ be a compact interval contained in 
$\real \setminus \{ \sigma _1 \cup \sigma _2 \} $.
If $P$ denotes the family of spectral projections for ${\cal L}$, then 
\cite[p. 237,264]{RS2} for any $f \in L^2({\cal T})$
\begin{equation}
\frac{1}{ 2}[P_{[a,b]} + P_{(a,b)}]f = 
\lim_{\epsilon \downarrow 0} \frac{1}{ 2 \pi i}\int_a^b [R(\lambda  + i\epsilon ) 
- R(\lambda  - i\epsilon )]f \ d\lambda \,. \label{4.c}
\end{equation} 
By the observations above, the right hand side of (\ref{4.c}) vanishes on the dense set
of $f$ supported on finitely many edges.  This means that
$[P_{[a,b]} + P_{(a,b)}]f =0$ for all $f \in L^2({\cal T})$,
and $[a,b]$ is in the resolvent set $\rho $. 
Thus $\complex \setminus \{ \sigma _1 \cup \sigma _2 \} \subset \rho $.

\end{proof}

\section{Closed walks and resolvent traces}

In this section the analysis of the resolvent of the biregular
tree $R_{\cal T}(\lambda  )$ is extended to  the resolvent $R_{\cal G}(\lambda  )$ of a general biregular graph ${\cal G}$.  
$\tree $ and $\graph $ are linked via the theory of covering spaces; see \cite{Massey} and \cite{Hatcher} for this material are sources for the theory and its application to graphs.
A feature of the approach using covering spaces is that generating functions for closed nonbacktracking walks in $\graph $
arise naturally.

\subsection{Walks} 

A combinatorial walk of length $l$ in $\graph $ (or $\tree$) starting at the vertex $v_1$ and ending at $v_{l+1}$
is a sequence $v_1,e_1,v_2,e_2,\dots ,e_{l},v_{l+1}$ with 
adjacent vertices $v_i,v_{i+1}$ joined by edges $e_i$.  A walk is closed if $v_1 = v_{l+1}$, and nonbacktracking if consecutive undirected edges are distinct.
A closed walk may be nonbacktracking if the edge $\{ v_1,v_2 \}$ is the same as  $\{v_1,  v_{l+1}\}$; in the terminology of \cite{Terras} tails are allowed.

If $\eta _l$ denotes the number of nonbacktracking closed walks of length $l \ge 1$ in $\graph$, then  $p_{\graph}(z) = N_{\eset} + \sum_{l=1}^{\infty} \eta _lz^l$ 
will be the associated     
generating function.  When $\graph $ is finite with  adjacency matrix $A$, an elementary result in graph spectral theory 
gives the number ${\widetilde \eta _l}$ of closed walks of length $l$ (with backtracking allowed) as the trace of the $l-th$ power of $A$.
The comparison $\eta _l \le {\widetilde \eta _l}$ can be used to show that $p_{\graph}(z)$ converges to an analytic function in a neighborhood of $z = 0$.

Using an identification of the edges with intervals of unit length, there is an edgepath $\gamma :[0,l] \to \graph $
which traverses the edges $e_i$ of a walk in order at unit speed.
Any loop in $\graph$ with basepoint $v_1$ is homotopic  \cite[p. 86]{Hatcher} to such an edgepath. 
By adding vertices of degree $2$ (without changing the lengths of the original edges), 
closed edgepaths may be assumed to start and end at edge midpoints.

It will be useful to determine the walk generating function $p_{\graph}(z)$ for the complete bipartite graphs $K(m_B,m_R)$ with $m_B$ vertices in $\vset _B$ and
$m_R$ vertices in $\vset _R$.
Closed walks in bipartite graphs have even length, so assume $l = 2k$ is even.  Start by counting the number of
nonbacktracking walks of length $2k-1$ starting at $v \in \vset _R$ and ending in $\vset _B$.  The vertex $v$ may be followed by any of $m_B$ vertices in $\vset _B$.
To avoid backtracking, there are then $m_R -1$ available vertices in $\vset _R$, then $m_B-1$ available vertices in $\vset _B$, and so on.  
After $2k-1$ steps the number of such walks is $m_B(m_R-1)^{k-1}(m_B-1)^{k-1} $.  The last step must return to $v$, but to avoid backtracking 
the walks which returned to $v$ after $2k-2$ steps are omitted.  The count $C_k$ of closed nonbacktracking walks of length $2k$ starting at $v$
satisfies $C_k = m_B(m_R-1)^{k-1}(m_B-1)^{k-1} - C_{k-1}$.
Multiplying by the number of starting $R$ vertices and doing the same for starts in $\vset _B$ shows that 
the total number $W_k$ of  closed nonbacktracking walks of length $2k$ with designated starting vertex satisfies 
\[W_k = 2m_Rm_B(m_R-1)^{k-1}(m_B-1)^{k-1} - W_{k-1}, \quad k \ge 2, \quad W_1 = 0.\]

That is, for $k \ge 2$,
\[W_k = 2m_Rm_B(-1)^{1-k} \sum_{j = 1}^{k-1} (-1)^{j} (m_R-1)^j(m_B-1)^j \]
\[= \frac{2m_Rm_B(m_R-1)(m_B-1)} {1 + (m_R-1)(m_B-1)} [(m_R-1)^{k-1}(m_B-1)^{k-1} + (-1)^{k}] \]
Adding the number of edges for order zero term gives the generating function
\begin{equation} \label{BCG}
P_{CG}(z) =  \frac{m_Rm_B}{2} + \frac{2m_Rm_B(m_R-1)(m_B-1)} {1 + (m_R-1)(m_B-1)} \times 
\end{equation}
\[\times \Big[ \sum_{k=2}^{\infty} (m_R-1)^{k-1}(m_B-1)^{k-1} z^{k} + \sum_{k=2}^{\infty} (-1)^kz^{k} \Big] \]
\[ =  \frac{m_Rm_B}{2} +  \frac{2m_Rm_B(m_R-1)(m_B-1)} {1 + (m_R-1)(m_B-1)}\Big[ \frac{(m_R-1)(m_B-1)z^2}{1 - (m_R-1)(m_B-1)z}  + \frac{z^2}{1 + z} \Big ]. \]
$W_k$ is a count of walks of length $2k$, so the closed nonbacktracking walk generating function is 
\begin{equation} \label{bCG}
p_{CG}(z) = P_{CG}(z^2) =
\end{equation}
\[\frac{m_Rm_B}{2} +  \frac{2m_Rm_B(m_R-1)(m_B-1)} {1 + (m_R-1)(m_B-1)}\Big[ \frac{(m_R-1)(m_B-1)z^4}{1 - (m_R-1)(m_B-1)z^2}  + \frac{z^4}{1 + z^2} \Big ]. \]

Suppose $\{r,b \}$ is an edge of $K(m_B,m_R)$, or any bipartite graph, with midpoint $w$.  
The closed nonbacktracking walks have one walk starting at $b$ and followed by $r$, and one
starting at $r$ and followed by $b$.  Without changing the count, these walks can be viewed as starting at $w$ followed by $b$ or $r$.  This alternate view 
will be useful later.

\subsection{Coverings}

Each biregular graph $\graph $ has a universal covering space $({\tree},p)$, 
where as before $\tree $ is the biregular tree.   The continuous map $p:\tree \to \graph $ 
is such that for each $x \in \graph $ there is an open neighborhood $N$ 
containing $x$ such that $p^{-1}(N)$ is a union of pairwise disjoint sets, each homeomorphic to $N$.
The $R,B$ labeling of tree vertices can be chosen so that $p$ preserves vertex class.

Pick basepoints $\xi _0 \in \graph $ and $\widetilde \xi _0 \in p^{-1}(\xi _0)$.  If $\widetilde \gamma $ is a nonbacktracking edgepath
in $\tree $ starting at  $\widetilde \xi _0$ and ending at $\widetilde \xi _1 \in p^{-1}(\xi _0)$, then 
$ \gamma = p(\widetilde\gamma )$ will be a closed nonbacktracking edgepath in $\graph $.  
Every closed nonbacktracking edgepath arises in this way \cite[p. 86]{Hatcher}.

Suppose that $\xi _0$ is a point in the interior of the edge $e_0 \in {\cal G}$,
and that $\widetilde{\xi}_0 \in p^{-1}(\xi _0)$.  Let $\widetilde e_0$ be the edge of ${\cal T}$
containing $\widetilde{\xi}_0$.  Then given any function $f \in L^2(e_0)$, there is a 
corresponding function $\widetilde f \in L^2(\tree )$ such that
\[ \widetilde f(\widetilde{\xi} ) = \left\{ \begin{array}{ll}
 f(p(\widetilde{\xi} )), & \widetilde{\xi} \in \widetilde e_0\,, \\ 
0, & \widetilde{\xi} \notin \widetilde e_0 \end{array}\right\}. \]

\begin{prop}\label{Thm5.1}  Suppose that $f \in L^2({\cal G})$ has support on an edge $e_0$.
There is a $C > 0$ such that if
$|Im(\sqrt \lambda  )| > C$ then for $\xi \in {\cal G}$ 
$$[R_{\cal G}(\lambda  )f](\xi ) 
= \sum_{\widetilde{\xi} \in p^{-1}(\xi )} [R_{\cal T}(\lambda  )\widetilde f](\widetilde \xi ).$$  
The sum and its first two derivatives converge uniformly for $\xi \in {\cal G}$. 
\end{prop}

\begin{proof} The proof has two parts: a formal verification and a proof that the sum converges.
Consider the two sums
$$H(\xi ,\lambda  ) = \sum_{\widetilde{\xi} \in p^{-1}(\xi )} [R_{\cal T}(\lambda  )\widetilde f](\widetilde \xi ), \quad 
h(\xi ,\lambda  ) = \sum_{\widetilde{\xi} \notin \widetilde e_0} [R_{\cal T}(\lambda  )\widetilde f](\widetilde \xi ).$$
Note that 
$$(- D^2 - \lambda  ) H(\xi ,\lambda  )
= \left\{ \begin{array}{ll}
 f(\xi ) , & \xi \in  e_0, \\
0, &\xi \notin e_0. \end{array}\right.
$$  
Since the vertex conditions are satisfied in the tree, they are still 
satisfied when we sum over vertices in the tree.

To check on the convergence of $H(\xi ,\lambda  )$ it suffices to check $h(\xi ,\lambda  )$.
Since the local homeomorphism of edges extends to edge pairs $\{r_n,b_n\}, \{b_n,r_{n+1}\}$,  
uniform convergence of the series for the $k-th$ derivative, $k = 0,1,2$ of $h(\xi ,\lambda  )$
is implied by  
$$\sum_{n=0}^\infty ( \delta _B \delta _R |\mu ^-|)^n  < \infty , $$
or $ \delta _B \delta _R|\mu ^-(\lambda  )| < 1 $.
For $C$ sufficiently large, this holds for  $|\,{\rm Im}\,(\sqrt \lambda  )| > C$ 
by \lemref{asyneg}.

If ${\cal G}$ is not a finite graph we must still check that $h \in L^2(\graph )$.
Let $e_a \in \graph$ denote an edge pair as above.  It will suffice
to consider the summands of $h$ with $\widetilde \xi  \in \tree _0$.
This contribution to the $L^2(\graph )$ norm is dominated by a sum
\[ \sum_{e_a \in {\cal G}} \int_0^2 |V(x,\lambda  )|^2 | \sum_{\widetilde{e_a} \in p^{-1}(e_a)} (\mu ^-)^{k(m)/2} |^2, \]
where the sum is taken over distinct edge pairs $e_a$, and
$k(m)$ (which may be taken even) is the distance from the tree root $r_0$ to $\widetilde{e_a}$.

This sum converges with
$$\sum_{e_a}|S_{e_a}|^2, 
\quad S_{e_a} = \sum_{\widetilde{e_a} \in p^{-1}(e_a)} (\mu ^-)^{k(m)}.$$
If the largest magnitude of a term in $S_{e_a}$ is $|\mu ^-|^n$, then
$$|S_{e_a}| \le \sum_{j=n}^{\infty} (\delta _R\delta _B |\mu ^-|)^j 
\le \frac{\delta _R\delta _B|\mu ^-|^n}{ 1 - \delta _R\delta _B |\mu ^-|}.$$
There are at most $(\delta _R\delta _B)^n$ sums $S_{e_a}$
with a leading term whose magnitude is as large as $|\mu ^-|^n$.
This count gives the bound 
$$\sum_{e_a}|S_{e_a}|^2 \le 
\sum_n (\delta _R\delta _B)^n\Bigl ( \frac{(\delta _R\delta _B)^n|\mu ^-|^n}{ 1 - \delta _R\delta _B |\mu ^-|} \Bigr )^2  
\le \frac{1}{ [1 - \delta _R\delta _B |\mu ^-|]^2} \sum_n (\delta _R\delta _B)^{3n}|\mu ^-|^{2n}.$$
Thus $h$ is square integrable if    
$(\delta _R\delta _B)^{3}|\mu ^-|^{2} < 1$.  

After modifying $C$, it follows that $H(\xi ,\lambda  ) $ is in the domain of $\cal L$,
with $(-D^2 -\lambda  )H(\xi ,\lambda  ) = f$.  

 \end{proof}

The diagonal of the resolvent kernel and the trace of the resolvent of a finite biregular graph $\graph $
contain information about closed walks in $\graph $. 
Let $e$ be an edge of the biregular graph $\graph $ with midpoint $g$ and vertices $r$ and $b$ of $e$ identified with $[0,1]$ as above. 
Define $P_e(z) = 1 + \sum_{l > 0} \eta _l z^{l/2}$ with coefficients $\eta _l$ counting  closed nonbacktracking walks of length $l$ starting at $g$.
Recall that $\eta _l = 0$ if $l$ is odd.

\begin{thm}\label{Thm5.2}    
For $|\,{\rm Im}\,(\sqrt \lambda  )| > C$ and $t \in e$ the diagonal of the resolvent may be written as
\begin{equation}
R_{\cal G}(t,t,\lambda  ) = P_e(\mu ^-(\lambda )) \frac{U(t,\lambda  )V(t,\lambda  )}{ -W(\lambda  )}, \quad t \in e . 
\label{5.a}\end{equation}

If ${\cal G}$ is a finite biregular graph with $N_{\eset}$ edges, then 
\begin{equation} \label{mainform}
{\rm tr}R_{\cal G}(\lambda  ) = \frac{P_{\graph }(\mu ^-(\lambda ))}{-W(\lambda  )} 
\int_0^1 U(t,\lambda  )V(t,\lambda  ) \ dt , 
\end{equation}
The function $P_{\graph}(z) = \sum_e P_e(z) = N_{\eset} + \sum_{l > 0} \eta _l z^{l/2}$, where
$\eta _l$ counts closed nonbacktracking walks of length $l$,
with one of the $N_{\eset}$ basepoints at the midpoint of an edge. 
\end{thm}

\begin{proof} 
Fix $\widetilde g \in \tree$, identifying the edge containing $g$ with $[0,1]$  
and let the function $x:\tree \to \real$ be as described above.
  
Suppose first that $\xi \in e$, where $f$ is supported.   
Split the resolvent sum into three parts,
\begin{equation}
[R_{\cal G}(\lambda  )f](\xi ) 
= [R_{\cal T}^0(\lambda  )\widetilde f](\xi ) + 
[R_{\cal T}^+(\lambda  )\widetilde f](\xi ) + 
[R_{\cal T}^-(\lambda  )\widetilde f](\xi ), \label{5.c}
\end{equation}  
where the $0,+,-$ terms are the resolvent sums of Theorem~\ref{Thm5.1} coming respectively
from $\widetilde \xi \in \widetilde e$, $x(\widetilde \xi ) > 1$ and $x(\widetilde \xi ) < 0$.
In the following sums, $e_m \in p^{-1}(e)$ means that there is a closed edgepath $\gamma $ starting at $g$ which lifts to 
a path $\widetilde{\gamma}$ based at $\widetilde g$ and ending in $e_m$.  Let $l(m) = 2k$ be the length of $\widetilde{\gamma}$. 
The three terms are given by integration against kernels, with $0 \le t \le 1$.
$R_{\cal T}^0(x,t,\lambda  )$ is given by \eqref{resform}, while 
\begin{equation}
R^+(x,t,\lambda  ) = \frac{U(t,\lambda  )V(x,\lambda  )}{ -W(\lambda  )} 
\sum_{x(e_m) > 1} [\mu ^-(\lambda  )]^{l(m)/2}, \quad e_m \in p^{-1}(e), \label{5.d}\end{equation}
and  
\begin{equation}
R^-(x,t,\lambda  ) = \frac{V(t,\lambda  )U(x,\lambda  )}{ -W(\lambda  )} 
\sum_{x(e_m) < 0} [\mu ^-(\lambda  )]^{l(m)/2}, \quad e_m \in p^{-1}(e). \label{5.e}
\end{equation}
That is, a closed nonbacktracking walk of length $l$ contributes powers $ \mu ^-(\lambda  )^{l/2} $. 
In case $\xi \notin e$ the term $R_{\cal T}^0$ will be missing, but otherwise
the representation of the resolvent will have the same form.    

Introduce the functions 
\[P^+_e(z) = \sum_{x(e_m) > 1} z^{l(m)/2} = \sum_{l > 0} \eta ^+_l z^{l/2},
\quad P^-_e(z) = \sum_{x(e_m) < 0} z^{l(m)/2} = \sum_{l > 0} \eta ^-_l z^{l/2},\]
and $P_e(z) = 1 + P^+_e(z) + P^-_e(z)$.
The coefficients $\eta ^+_l $, $\eta ^-_l $, count closed edgepaths in $\graph $ based 
at the midpoint of $e$, and whose lift from the midpoint of $\tilde e$ to the midpoint of $\tilde e$
has length $l$, with $x(\tilde e)$ respectively greater than $1$ or less than $0$.  
Setting $x=t$ and including the three terms in (\ref{5.c}) gives the description of the 
diagonal of the resolvent in the statement of the theorem.   \end{proof} 

Notice that \eqref{mainform} expresses ${\rm tr}R_{\cal G}(\lambda  )$ as a product, with the factor  
\[  \frac{1}{W(\lambda  )} \int_0^1 U(t,\lambda  )V(t,\lambda  ) \ dt \]
independent of $\graph $.  Dividing the general case of \eqref{mainform} by the example of the complete bipartite 
graph with the same degrees gives the corollary    
\begin{equation} \label{key1}
\frac{P_{\graph }(\mu ^-(\lambda ))} {P_{CB}(\mu ^-(\lambda ))} = \frac{{\rm tr}R_{\cal G}(\lambda  )} {{\rm tr}R_{CB}(\lambda  )}.
\end{equation}

\subsection{Structure of $P_{\graph }(z)$}

The effectiveness of \eqref{key1} is based on our knowledge of three of the four functions which appear.
The resolvent trace ${\rm tr}R_{\cal G}(\lambda  )$ is computed in \eqref{trace1b} from the eigenvalues of the discrete Laplacian $\Delta $.
The special case ${\rm tr}R_{CB}(\lambda  )$ is then presented in \eqref{comptrace}.  For the complete bipartite graph
$m_B = \delta _R +1$ and $m_R = \delta _B +1$, so \eqref{bCG} becomes
\begin{equation} \label{bCGform} 
P_{CG}(\mu ^-(\lambda )) =  \frac{(\delta _B +1)(\delta _R +1)}{2} 
\end{equation}
\[+  \frac{2(\delta _B +1)(\delta _R +1)\delta _B\delta _R} {1 + \delta _B\delta _R}
\Big[ \frac{\delta _B\delta _R\mu ^-(\lambda )^2}{1 - \delta _B\delta _R\mu ^-(\lambda )}  + \frac{\mu ^-(\lambda )^2}{1 + \mu ^-(\lambda )} \Big ] \]
The tree based term in \eqref{mainform} satisfies
\[-\frac{1}{W(\lambda  )} \int_0^1 U(t,\lambda  )V(t,\lambda  ) \ dt = {\rm tr}R_{CB}(\lambda  )/P_{CB}(\mu ^-(\lambda ))  . \]

From \eqref{BCG} it is easy to verify that the generating function $P_{CG}(z)$ in series form converges to an analytic function
for $|z| < \frac{1}{\delta _B\delta _R}$, and that it extends to a rational function for $z \in \complex $.
Currently, the function $P_{\graph }(\mu ^-(\lambda ))$ is only defined on the range of $\mu ^-(\lambda )$.
As described in \thmref{dom1}, $| \mu ^-(\lambda )| < 1/\sqrt{\delta _B\delta _R}$.  
Since $P_{CG}(z)$ has a rational extension, \eqref{key1} can be used to extend $P_{\graph }(z)$
by taking $z = \Psi (\lambda )$, where
\begin{equation} \label{zdef}
z = \Psi (\lambda ) = \Bigl \{ \begin{matrix} \mu ^{-}(\lambda), & |z| < 1/\sqrt{\delta _R\delta _B}, \cr
 \mu ^{+}(\lambda), & |z| > 1/\sqrt{\delta _R\delta _B}. \end{matrix} \Bigr \} 
 \end{equation}

The trace formula
\begin{equation} \label{eigrel}
\mu ^+(\lambda ) + \mu ^-(\lambda ) =   \mu ^- + 1/(\delta _R \delta _B \mu ^-) = {\rm tr}(T_j(\lambda )) 
\end{equation}
leads to consideration of the mapping $\phi (z) = z + 1/(c z)$, where $z \in \complex \setminus \{0\}$ and $c > 0$.
This discussion follows \cite[p. 89]{Marshall}, where additional information is available.
If $w = z + 1/(c z)$, then $2z = w \pm \sqrt{w^2 - 4/c}$, and $\phi (z)$ is two to one unless $w = \pm 2c^{-1/2}$. 
Since $w \pm \sqrt{w^2 - 4/c} \not= 0$ for $w \in \complex$, $\phi (z)$ is also surjective.  

In addition, $\phi (z) = \phi (1/(cz))$, so $\phi $ is injective for $|z| < c^{-1/2}$ and $|z| > c^{-1/2}$.  
The image of the circle $z = c^{-1/2}e^{i \theta }$ is the interval $[-2c^{-1/2}, 2c^{-1/2}]$.
Thus $\phi : \{ 0 < |z| < c^{-1/2} \} \to \complex \setminus [-2c^{-1/2}, 2c^{-1/2}] $ and $\phi : \{ |z | > c^{-1/2}  \} \to \complex \setminus [-2c^{-1/2}, 2c^{-1/2}] $ 
are both conformal maps.

To extend the definition of $z$ in \eqref{zdef}, first note that ${\mathcal C} = \{ (\lambda , \mu ^{\pm}(\lambda )) \in \complex ^2 \}$
is a Riemann surface presented as a two-sheeted branched cover of $\complex $.  As long as the eigenvalues of $T_j(\lambda )$ are distinct,
$\Psi (\lambda ) = \mu (\lambda )$ is a locally defined analytic function.  The proof of \thmref{dom1} shows that ${\mathcal C} $ is connected.
The definition of $z$ in \eqref{zdef} may be extended as the value of $\mu $ as a function from ${\mathcal C}$ to $\complex $.

\begin{lem} \label{psidef}
The extended map $z = \mu $ taking ${\mathcal C} $ to $\complex \setminus \{ 0 \}$ is surjective,
with the locally defined $\mu = \Psi (\lambda )$ satisfying $\Psi '(\lambda ) \not= 0$ unless $\cos (\omega )\sin (\omega )/\omega  = 0$.
\end{lem}

\begin{proof}
From \eqref{trace0},
\begin{equation} \label{traced}
{\rm tr}T_j(\lambda )
= \cos ^2(\omega )(1 + \frac{1}{\delta _B})(1 + \frac{1}{\delta _R}) - (\frac{1}{\delta _R} + \frac{1}{\delta _B} ).
\end{equation}
One checks easily that the range of ${\rm tr}T_j(\lambda )  $ is $\complex $.

The comments above about the mapping $\phi (z)$ mean that all pairs $\{\mu ^-,  \mu ^+ = 1/(\delta _R\delta _B\mu^-)\}$ with $\mu ^- \not= 0$ 
occur for some values of $\lambda $.   With $|\mu ^-| <  1/\sqrt{\delta _R\delta _B}$ and $|\mu ^+| > 1/\sqrt{\delta _R\delta _B}$ 
the individual values of $\mu ^-(\lambda )$ and $\mu ^{+}(\lambda ) $ are well-defined and analytic. 
The values of $\mu ^-(\lambda )$ and $\mu ^+(\lambda )$ respectively cover the sets $\{0 < |z| < 1/\sqrt{\delta _R\delta _B} \} $ and 
$\{|z| > 1/\sqrt{\delta _R\delta _B} \}$, showing that the range of $\Psi $ includes $\complex \setminus \{0,|z| = 1/\sqrt{\delta _R\delta _B}\}$.

It remains to consider $|\mu ^{\pm}(\lambda )| = 1/\sqrt{\delta _R\delta_B}$, in which case $\lambda \in [0,\infty )$ by \lemref{cross}.
As $\cos ^2(\omega ) $ decreases from $1$ to $0$ in \eqref{traced},  ${\rm tr}T_j(\lambda )$ decreases  from   
$1 + \frac{1}{\delta _B}\frac{1}{\delta _R}$ to $- (\frac{1}{\delta _R} + \frac{1}{\delta _B} )$.  From $(1/\sqrt{\delta _R} - 1/\sqrt{\delta _B} )^2 \ge 0$ follows
$1/\delta _R + 1/\delta _B \ge 2/\sqrt{\delta _r\delta _B}$ and from $(1 - 1/\delta _R)(1 - 1/\delta _B) \ge 0$
follows $1 + \frac{1}{\delta _B}\frac{1}{\delta _R} \ge (\frac{1}{\delta _R} + \frac{1}{\delta _B} )$.  Thus
\[ - (\frac{1}{\delta _R} + \frac{1}{\delta _B} ) \le -2/\sqrt{\delta  _R\delta _B} <  2/\sqrt{\delta  _R\delta _B} \le 1 + \frac{1}{\delta _B}\frac{1}{\delta _R} .\]
The behavior is similar as $\cos ^2(\omega ) $ increases from $0$ to $1$.  From \eqref{quadform} and \eqref{eval1}
the eigenvalues $\mu ^{\pm}(\lambda )$ traverse the upper and lower semicircles $|\mu ^{\pm}(\lambda )| = 1/\sqrt{\delta _R\delta _B}$
once in each of these segments where $|{\rm tr}T_j(\lambda)| \le 2/\sqrt{\delta  _R\delta _B}$, finishing the surjectivity of $\mu $.
By \thmref{dom1} continuation of $\mu ^{\pm}(\lambda )$ across an interval with $|{\rm tr}T_j(\lambda)| < 2/\sqrt{\delta  _R\delta _B}$
is analytic as $\mu ^{\pm}(\lambda )$ switches to the $\mu ^{\mp}(\lambda )$ sheet.

Computing the derivative of ${\rm tr}T_j(\lambda )$ in two ways gives 
\[\frac{d {\rm tr} T_j(\lambda )}{d\lambda } = - \cos (\omega )\frac{\sin (\omega )}{\omega }(1+\frac{1}{\delta _R})(1 + \frac{1}{\delta _B}), \]
and 
\[\frac{d}{d \lambda } (\mu ^- + \frac{1}{\delta _R\delta _B \mu ^-}) = (\mu ^-)' (1 - \frac{1}{\delta _R\delta _B (\mu ^-)^2} )\]
so zeros of the derivative of $\mu ^{-}(\lambda )$ are constrained by
\begin{equation} \label{dercomp}
\mu ^-(\lambda )' (1 - \frac{1}{\delta _R\delta _B (\mu ^-(\lambda ))^2} ) =  - \cos (\omega )\frac{\sin (\omega )}{\omega }(1+\frac{1}{\delta _R})(1 + \frac{1}{\delta _B}).
\end{equation}
From $\mu ^+(\lambda ) \mu ^-(\lambda ) = 1/(\delta _R\delta _B)$,  
\[\mu ^+(\lambda )' = - \mu ^-(\lambda )'\mu ^+(\lambda )/ \mu ^-(\lambda ) ,\]
so the derivatives of $\mu ^+$ and $\mu ^-$ vanish together. 

\end{proof}

As noted, $\Psi (\lambda ) = \mu $ is a locally defined analytic function on the connected
set $\{ \mu ^+(\lambda ) \not= \mu ^-(\lambda ) \} $, so locally \eqref{key1} extends to 
\begin{equation} \label{key2}
\frac{P_{\graph }(\Psi (\lambda ))} {P_{CB}(\Psi (\lambda ))} = \frac{{\rm tr}R_{\cal G}(\lambda  )} {{\rm tr}R_{CB}(\lambda  )}.
\end{equation}  
or globally 
\begin{equation} \label{key3}
\frac{P_{\graph }(\mu )} {P_{CB}(\mu )} = \frac{{\rm tr}R_{\cal G}(\lambda  )} {{\rm tr}R_{CB}(\lambda  )}.
\end{equation} 
The left side of \eqref{key2} depends on $\mu = \Psi (\lambda )$, so that is true of the right side as well, implying that
$P_{\graph }(z)$ is well defined.  The right side of \eqref{key2} depends on $\lambda $, so 
$P_{\graph }(\mu ^+(\lambda ))/P_{CB}(\mu ^+(\lambda )) =  P_{\graph }(\mu ^-(\lambda ))/P_{CB}(\mu ^-(\lambda ))$.

\begin{thm}
If $\graph $ is a finite biregular graph then the generating function $P_{\graph}(z)$ extends analytically to 
\begin{equation} \label{xdom1}
\{ |z | < 1/\sqrt{\delta _R\delta _B} \} \setminus \{ z \in \real \ | \ 1/(\delta _R \delta _B)  \le |z|  < 1/\sqrt{\delta _R\delta _B}  \} 
\end{equation}
and 
\[\{ |z | > 1/\sqrt{\delta _R\delta _B} \} \setminus \{ z \in \real \ | \ 1/\sqrt{\delta _R\delta _B}  <  |z|  \le 1  \} ,\]
except for poles at $z = \pm i $.
\end{thm}

\begin{proof}

The resolvent traces in \eqref{key2} are meromorphic in $\complex$, with poles in $[0,\infty )$.  
As noted in \propref{reality}, ${\rm tr}R_{CB}(\lambda  ) = 0$ implies $\lambda \ge 0$.
Thus $Q(\lambda ) = {\rm tr}R_{\cal G}(\lambda  )/ {\rm tr}R_{CB}(\lambda  )$ is meromorphic in $\complex$, with all poles 
in $[0,\infty )$. 
By \thmref{dom1} $\mu ^-(\lambda )$ and $\mu ^+(\lambda)$ are analytic in the complement of $\sigma _1 \subset [0, \infty )$.
Since $P_{CG}$ is rational, $P_{\graph }(\mu ^-(\lambda ))$ and $P_{\graph }(\mu ^+(\lambda ))$ are meromorphic in $\complex \setminus \sigma _1$.

By \eqref{dercomp}, $\mu ^-(\lambda )' $ is not zero unless $\cos (\omega )\sin (\omega )/\omega = 0$;
these roots only occur for $\lambda \ge 0$.  If $\mu ^-(\lambda _0)' \not= 0$, then $\mu ^+(\lambda _0 )' \not= 0$, 
in which case both $\mu ^+(\lambda )$ and $\mu ^-(\lambda )$ are conformal maps in a neighborhood of $\lambda _0$. 
If $Q(\lambda ) $ is analytic in a neighborhood of $\lambda _0$ and if $P_{CB}(\Psi(\lambda ))$ is also analytic in a neighborhood of $\mu ^{\pm}(\lambda _0)$, 
then \eqref{key2} provides a definition of $P_{\graph}(z) = P_{\graph }(\Psi(\lambda ))$ as an analytic function in some neighborhood of $z_0 = \mu ^{\pm}(\lambda _0)$.  

$Q(\lambda )$ is analytic in $\complex \setminus [0,\infty )$.  As noted for \eqref{eval1}, ${\rm tr}(T_j)(\lambda)$ is real
when $\lambda  $ is real.   If $\lambda \in \real $ and ${\rm tr}(T_j)^2/4 - \det (T_j) \le 0$ then $|\mu ^{\pm}(\lambda ) | = 1/\sqrt{\delta _B \delta _R}$.
If $\lambda \ge 0$ and ${\rm tr}(T_j)^2/4 - \det (T_j) \ge 0$ then $\mu ^{\pm}(\lambda ) \in \real $ and 
$|{\rm tr}(T_j)(\lambda ) | \le 1 + \frac{1}{\delta _R}\frac{1}{\delta _B}$ by \eqref{traced}.  By \eqref{quadform} 
\[\mu ^+ \le \frac{1}{2}(1 + \frac{1}{\delta _R}\frac{1}{\delta _B}) + \sqrt{(1 - \frac{1}{\delta _R}\frac{1}{\delta _B})^2/4} = 1,\]
so $\mu ^-(\lambda ) \ge 1/(\delta _R \delta _B)$ for $\lambda \ge 0$.  
Finally, the rational function $P_{CB}(z) $ has poles at $z = \pm i , \pm 1/(\delta _R\delta _B)$.
Thus $P_{\graph} (z) $ is analytic in the sets in \eqref{xdom1}.
\end{proof}

By \eqref{traced} the functions ${\rm tr}T_j(\lambda )$, and so the (unordered pair of) eigenvalues $\mu ^{\pm}(\lambda )$, 
are periodic in $\omega $ with period $\pi $.  Thus any singularities in $P_{\graph }(\mu )$ must already appear in 
the strip $0 \le \Re (\omega ) \le \pi $.  The function $Q(\lambda ) = {\rm tr}R_{\cal G}(\lambda  )/ {\rm tr}R_{CB}(\lambda  )$ only has singularities in $[0, \infty )$,
meaning that singular points for $P_{\graph }(z)$ coming from $Q(\lambda )$ will appear for $0 \le \lambda \le \pi ^2$,
where $Q(\lambda )$ has a finite set of poles.  

By \eqref{trace1b} the pole locations for ${\rm tr}R_{\cal G}(\lambda  )$ with $0 \le \lambda \le \pi ^2$ are $0$, $\pi ^2$ 
and the points with $\cos (\sqrt{\lambda } ) = 1 - \nu _k$ for $k = 1,\dots ,N_{\vset } - 2$. 
For the complete bipartite graph $\dop $ has eigenvalues at $0,(\pi/2 )^2, \pi ^2$ in the interval $[0, \pi ^2]$. 
By \propref{reality} the function ${\rm tr}R_{CG}(\lambda )$ has two roots $\xi _i$ lying between the consecutive eigenvalue pairs.
Thus in the strip $0 \le \Re (\lambda ) \le \pi ^2$ the function $Q(\lambda ) $ has at most $N_{\vset} + 2$ poles, all lying in $[0,\pi ^2 ]$. 
 
\begin{thm}
If $\graph $ is a finite biregular graph then the generating function $P_{\graph}(z)$ extends to a 
rational function.  The poles are all located in the set
\[\{ z = \pm i\} \cup \{ |z| = 1/\sqrt{\delta _R\delta _B} \} \cup \{ z \in \real \ | \ 1/(\delta _R \delta _B)  \le |z|  < 1  \} .\]
\end{thm}

\begin{proof}

The function $P_{\graph}(z) $ is analytic except possibly at a finite collection of points for which:

(i) $\lambda $ is a pole of $Q(\lambda )$, 

(ii)  $\mu ^{\pm}(\lambda )' = 0$, 

(iii) $\mu ^-(\lambda ) = \mu ^+(\lambda )$, 

(iv) $z$ is a pole of $P_{CG}(z)$. 

By \eqref{key2} the isolated singular points for $z \in \complex $ are not essential singularities, so \cite[p. 105-109]{GK} $P_{\graph }(z)$ is 
meromorphic in $\complex $ with a finite set of poles, which are limited to the sets (i) and (iv). In the strip $0 \le \Re (\omega ) \le \pi $, 
note that $|{\rm tr}T_j(\lambda )| \to \infty $, and so  $|\mu ^+(\lambda )| \to \infty $ as $| \Im (\omega )| \to \infty $.
From \eqref{trace1b} the resolvent traces have limit zero as $| \Im (\lambda )| \to \infty $ in this strip.
Moreover $P_{CB}(z)$ has a pole at $\infty $.  Thus  $P_{\graph }(z)$ is a meromorphic function with finitely many poles in $\complex \cup \infty $,
so is rational \cite[p. 141]{GK}.

\end{proof}


\begin{thebibliography}{10}

\bibitem{BDH}
G.~Brito, I.~Dumitiu, and K.~Harris.
\newblock{Spectral gap in random bipartite biregular graphs and its applications,}
\newblock{Combinatorics Probability and Computing},
31(2) 2018. 

\bibitem{Below}
J. ~von Below.
\newblock {A characteristic equation associated to an eigenvalue problem on $c^2$ - networks},
\newblock {Linear Algebra and Its Applications}, 
71(23): 309--3325, 1985.

\bibitem{BK}
G.~Berkolaiko and P~Kuchment.
\newblock {\em Introduction to Quantum Graphs}.
\newblock American Mathematical Society, Providence, 2013.


\bibitem{BR}
G.~Birkhoff and G~Rota.
\newblock {\em Ordinary Differential Equations}.
\newblock Blaisdell, Waltham, 1969.


\bibitem{Bollobas}
B. ~Bollobas. \textit{Modern Graph Theory.}
Springer, 1998.



\bibitem{Brooks91}
R.~Brooks.
\newblock {\em The spectral geometry of k-regular graphs}.
\newblock {\em J. d'Analyse Mathematique}, 57:120--151, 1991.

\bibitem{Carlson97}
R.~Carlson.
\newblock Hill's equation for a homogeneous tree.
\newblock {\em Electronic Journal of Differential Equations}, 1997(23):1--30, 1997.

\bibitem{CarlsonFFT}
R.~Carlson
\newblock{Harmonic Analysis for Graph Refinements and the Continuous Graph FFT}
\newblock{Linear Algebra Appl.} 430 no. 11-12: 2859–2876, 2009.

\bibitem{Cattaneo}
C. ~Cattaneo.
\newblock {The spectrum of the continuous Laplacian on a graph},
\newblock {Monatsh. Math.}, 
124(3): 215--235, 1997.

\bibitem{Chung}
F.~Chung.
\newblock {\em Spectral Graph Theory}.
\newblock American Mathematical Society, Providence, 1997.

\bibitem{GK}
R.~Greene and S.~Krantz.
\newblock {\em Function Theory of One Complex Variable}.
\newblock American Mathematical Society, 2006.

\bibitem{Hatcher}
A.~Hatcher.
\newblock {\em Algebraic Topology}.
\newblock Cambridge University Press, 2001.
\newblock{Online edition}

\bibitem{Kato}
T.~Kato.
\newblock {\em Perturbation Theory for Linear Operators}.
\newblock Springer-Verlag, New York, 1995.

\bibitem{Marshall}
D.~Marshall.
\newblock {\em Complex Analysis}.
\newblock Cambridge, 2019.


\bibitem{Massey}
W.~Massey.
\newblock {\em Algebraic Topology: An Introduction}.
\newblock Harcourt, Brace and World, New York, 1967.


\bibitem{RS1}
M.~Reed and B.~Simon
\newblock {\em Methods of Modern Mathematical Physics, 1}.
\newblock Academic Press, New York, 1972.

\bibitem{RS2}
M.~Reed and B.~Simon
\newblock {\em Methods of Modern Mathematical Physics, 2}.
\newblock Academic Press, New York, 1975.

\bibitem{Terras}
A.~Terras.
\newblock {\em Zeta Functions of Graphs}.
\newblock Cambridge University Press, Cambridge, 2011.

\end{thebibliography}
\end{document}